\documentclass[11pt,a4paper]{amsart}

\usepackage{a4wide}
\usepackage{amsfonts,amsthm,amsmath,amssymb}
\usepackage{amscd,color}
\usepackage{psfrag,graphicx,subfigure}
\usepackage{import} 
\usepackage[shortlabels]{enumitem}

\theoremstyle{plain}
\newtheorem{lem}{Lemma}[section]

\newtheorem{thm}[lem]{Theorem}

\newtheorem*{ThmA}{Theorem A}
\newtheorem*{ThmB}{Theorem B}
\newtheorem*{ThmC}{Theorem C}
\newtheorem*{PropD}{Proposition D}

\newtheorem{cor}[lem]{Corollary}

\theoremstyle{definition}
\newtheorem{defn}[lem]{Definition}
\newtheorem*{defn*}{Definition}
\newtheorem{ex}[lem]{Example}

\newtheorem*{ex*}{Example}

\newtheorem{rem}[lem]{Remark}
\newtheorem*{rem*}{Remark}
\newtheorem*{ack}{Acknowledgement}
\theoremstyle{remark}

\DeclareMathOperator{\dist}{dist}

\DeclareMathOperator{\area}{area}
\DeclareMathOperator{\capacity}{cap}
\DeclareMathOperator{\arctanh}{arctanh}

\newcommand{\C}{\mathbb C}
\newcommand{\D}{\mathbb D}
\newcommand{\clC}{\widehat{\C}}
\newcommand{\chat}{\widehat{\C}}

\newcommand{\R}{\mathbb R}
\newcommand{\Z}{\mathbb Z}
\newcommand{\N}{\mathbb N}

\newcommand{\bd}{\partial}
\renewcommand{\Re}{\textup{Re}}
\renewcommand{\Im}{\textup{Im}}
\renewcommand{\H}{\mathbb H}
\newcommand{\HH}{\mathbb H}

\newcommand{\calj}{\mathcal{J}}

\newcommand{\bdd}{\partial \D}
\newcommand{\A}{\mathcal A}

\begin{document}

\title{Escaping points in the boundaries of Baker domains}
\date{\today}

\author{Krzysztof Bara\'nski}
\address{Institute of Mathematics, University of Warsaw,
ul.~Banacha~2, 02-097 Warszawa, Poland}
\email{baranski@mimuw.edu.pl}

\author{N\'uria Fagella}
\address{Departament de Matem\`atica Aplicada i An\`alisi, Institut de Matem\`atiques de la 
Universitat de Barcelona (IMUB) and Barcelona Graduate School of Mathematics (BGSMath). 
 Gran Via 585, 08007 Barcelona, Catalonia, Spain}
\email{nfagella@ub.edu}

\author{Xavier Jarque}
\address{Departament de Matem\`atica Aplicada i An\`alisi, Institut de Matem\`atiques de la 
Universitat de Barcelona (IMUB), and Barcelona Graduate School of Mathematics (BGSMath).
 Gran Via 585, 08007 Barcelona, Catalonia, Spain}
\email{xavier.jarque@ub.edu}

\author{Bogus{\l}awa Karpi\'nska}
\address{Faculty of Mathematics and Information Science, Warsaw
University of Technology, ul.~Ko\-szy\-ko\-wa~75, 00-662 Warszawa, Poland}
\email{bkarpin@mini.pw.edu.pl}

\thanks{Research supported by the Polish NCN grant decision DEC-2012/06/M/ST1/00168, carried out at the University of Warsaw. The 
second and third authors were partially supported by the Spanish 
grants MTM2011-26995-C02-02 and MTM2014-52209-C2-2-P}

\bibliographystyle{amsalpha}

\begin{abstract} 
We study the dynamical behaviour of points in the boundaries of simply connected invariant Baker domains $U$ of meromorphic maps $f$ with a finite degree on $U$. We prove that if $f|_U$ is of hyperbolic or simply parabolic type, then almost every point in the boundary of $U$ with respect to harmonic measure escapes to infinity under iteration. On the contrary, if $f|_U$ is of doubly parabolic type, then almost every point in the boundary of $U$ with respect to harmonic measure has dense forward trajectory in the boundary of $U$, in particular the set of escaping points in the boundary of $U$ has harmonic measure zero. We also present some extensions of the results to the case when $f$ has infinite degree on $U$, including classical Fatou example. 
\end{abstract}

\maketitle

\section{Introduction and statement of the results} \label{sec:intro}

Let $f : \C \to \clC$ be a meromorphic map of degree larger than $1$ and consider the dynamical system generated by the iterates $f^n = f\circ \cdots \circ f$. The complex sphere is then divided into two invariant sets: the {\em Fatou set} $\mathcal F(f)$, which is the set of points $z\in\chat$, where the family of iterates $\{f^n\}_{n\geq 0}$ is defined and normal in some neighbourhood of $z$, and its complement, the {\em Julia set} $\mathcal J(f) = \chat \setminus \mathcal F(f)$, where chaotic dynamics occurs. We refer to \cite{bergweiler,carlesongamelin,milnor} for the basic properties of Fatou and Julia sets.

It is well known that for any polynomial of degree larger than one, the point at infinity is a super-attracting fixed point and the set of points whose orbits tend to infinity coincides with its immediate basin of attraction. Note that no point in the boundary of this basin tends to infinity under iteration. In the case of a transcendental map $f$, where infinity is no longer an attracting fixed point but an essential singularity, the \emph{escaping set} of $f$, defined as
\[
\mathcal I(f) = \{z \in \C: f^n(z) \text{ is defined for every } n \ge 0 \text{ and } f^n(z) \to \infty \text{ as } n \to \infty\},
\]
often exhibits much richer topology. In many cases, for instance, it contains a {\em Cantor bouquet} consisting of a Cantor set of unbounded curves (see for example \cite{aarts,bjl,dt}). For transcendental entire maps $f$ it is known that $\partial \mathcal I(f) =\mathcal J(f)$ \cite{ere-escaping}. 

Similarly as for polynomials, transcendental maps may also have components of the Fatou set (known as {\em Fatou components}) which are contained in $\mathcal I(f)$.  These may be {\em escaping wandering domains}, i.e.~non-preperiodic components for which the sequence $\{f^n\}_n$ tends  locally uniformly to infinity; or {\em Baker domains}, that is periodic Fatou components with the same property, either for $\{f^n\}_n$ (invariant case) or for $\{f^{kn}\}_n$, with some $k>1$. Baker domains are sometimes called ``parabolic domains at infinity'', although their properties do not always resemble those of parabolic basins (see Remark \ref{rem:parabolic}). We refer to  \cite{barfag,bfjk,berdrasin,bergzheng,faghen,henriksen,rippon,ripponstallard} for general background and results on Baker domains.

Considering examples of escaping wandering and Baker domains of different nature (see e.g.~\cite{rippon,barfag,faghen,faghen09}), it is natural to ask whether there are points in their boundaries which escape to infinity under iteration and if so, how large is the set of such points. This question was addressed by Rippon and Stallard in \cite{ripponstallard-escaping-wandering}, where they showed that almost all points (in the sense of harmonic measure) in the boundary of an escaping wandering component are in the escaping set. Their proof is also valid for some class of Baker domains (see Remark~\ref{rem:baranski}). Recently, in an inspiring paper \cite{ripponstallard-escaping}, they extended the result to the case of arbitrary univalent Baker domains (i.e.~ invariant simply connected Baker domains where $f$ is univalent) of entire maps.

Our goal in this paper is to extend the analysis to the case of finite degree invariant simply connected Baker domains $U$ for meromorphic maps $f$ (with some extensions to infinite degree), showing that there is a dichotomy in the dynamical behaviour of the boundary points of $U$, depending on the {\em type} of the domain in the sense of the Baker--Pommerenke--Cowen classification. We say that $f|_U$ is of {\em hyperbolic} (resp.~{\em simply parabolic} or {\em doubly parabolic}) type if the dynamics in $U$ is eventually conjugate to $\omega\mapsto a \omega$ with $a>1$ on the right half-plane $\H$ (resp.~to $\omega \mapsto \omega\pm i$ on $\H$, or to $\omega\mapsto \omega+1$ on $\C$). See Theorem~\ref{thm:cowen} for details. Equivalently, doubly parabolic Baker domains are those for which the hyperbolic distance $\varrho_U$ in $U$ between $f^n(z)$ and $f^{n+1}(z)$ tends to $0$ as $n \to \infty$ for $z \in U$ (see Theorem~\ref{thm:doubly}).

In this paper we show that the Rippon and Stallard result from \cite{ripponstallard-escaping} for univalent Baker domains remains valid in the case of finite degree invariant simply connected Baker domains $U$ as long as $f|_U$ is of hyperbolic or simply parabolic type, while in the remaining doubly parabolic case the iterations of the boundary points behave totally differently -- a typical point in the sense of harmonic measure has dense trajectory in the boundary of $U$, in particular it is not escaping to infinity. The precise statements are presented below as Theorems~A and~B. In the proofs, apart from methods used in \cite{ripponstallard-escaping}, we rely on the results by Aaronson \cite{aaronson2,aaronson3} and Doering--Ma\~n\'e \cite{doering-mane} on the ergodic theory of inner functions. 

If $U$ is an invariant simply connected Fatou component of a meromorphic map $f$, then the map
\[
g: \D \to \D, \qquad g = \varphi^{-1}\circ f \circ \varphi,
\]
where $\varphi: \D \to U$ is a Riemann map, is an inner function (see Definition \ref{def:inner}) with degree equal to the degree of $f|_U$. We call $g$ an inner function \emph{associated} to $f|_U$. If the degree of $g$ is finite, then it is a finite Blaschke product and extends to a rational map of the Riemann sphere. If the degree is infinite, $g$ has at least one \emph{singular point} in the unit circle $\partial \D$, i.e.~a point $\zeta$ with no holomorphic extension of $g$ to any neighbourhood of $\zeta$. 

If $U$ is an invariant simply connected Baker domain of $f$, the associated inner function $g$ has no fixed points in $\D$, and the Denjoy--Wolff Theorem (see Theorem \ref{thm:denjoy}) implies the existence of the {\em Denjoy--Wolff point} $p$ in the unit circle, such that every orbit of a point in $\D$ under iteration of $g$ converges to $p$. See Section~\ref{sec:prelim} for details.

By the Fatou Theorem, the Riemann map $\varphi$ extends almost everywhere to the unit circle in the sense of radial limits. We consider the harmonic measure on $\bd U$ defined to be the image under $\varphi$ of the normalized Lebesgue measure on the unit circle. 

In this paper we prove the following. 

\begin{ThmA}
Let $f: \C \to \clC$ be a meromorphic map and let $U$ be a simply connected invariant Baker domain of $f$, such that the degree of $f$ on $U$ is finite and $f|_U$ is of hyperbolic or simply parabolic type $($i.e. $\varrho_U (f^{n+1}(z), f^n(z)) \not\to 0$ as $n \to \infty$ for $z \in U)$. Then the set $\mathcal I(f) \cap \bdd$ of escaping points in the boundary of $U$ has full harmonic measure.

More generally, the statement remains true if instead of a finite degree of $f$ on $U$ we assume that the associated inner function $g = \varphi^{-1}\circ f \circ \varphi$, where $\varphi: \D \to U$ is a Riemann map, has non-singular Denjoy--Wolff point.
\end{ThmA}

\begin{rem} \label{rem:baranski}
If $U$ is a simply connected invariant Baker domain of $f$, such that the degree of $f$ on $U$ is finite and
\begin{equation}\label{eq:>0}
\limsup_{n\to \infty} \frac{|f^{n+1}(z)- f^n(z)|}{\dist(f^n(z), \bd U)} >
0
\end{equation}
for some $z \in U$ (in particular, if 
\[
\limsup_{n\to \infty}  \left|\frac{f^{n+1}(z)}{f^n(z)} - 1\right| > 0
\]
for some $z \in U$), then $U$ is of hyperbolic or simply parabolic type (see  Theorem~\ref{thm:doubly}). Therefore, Theorem~A implies in this case that  the set of escaping points in the boundary of $U$ has full harmonic measure. In \cite{ripponstallard-escaping-wandering}, Rippon and Stallard proved that  the statement remains true without the finite degree assumption, if \eqref{eq:>0} is replaced by a stronger condition, saying that there exist $z \in U$ and $K > 1$ such that
\begin{equation}\label{eq:K}
|f^{n+1}(z)| > K |f^n(z)|
\end{equation}
for every $n > 0$. In fact, their proof gives that the condition
\eqref{eq:K} can be replaced by
\[
\sum_{n = 0}^\infty \frac{1}{\sqrt{|f^n(z)|}} < \infty.
\]
\end{rem}

Our next theorem shows that an opposite situation arises when the Baker domain is of doubly parabolic type.

\begin{ThmB}
Let $f: \C \to \clC$ be a meromorphic map and let $U$ be a simply connected invariant Baker domain of $f$, such that the degree of $f$ on $U$ is finite and $f|_U$ is of doubly parabolic type $($i.e. $\varrho_U (f^{n+1}(z), f^n(z)) \to 0$ as $n \to \infty$ for $z \in U)$. Then the set of points $z$ in the boundary of $U$, whose forward trajectories $\{f^n(z)\}_{n \ge 0}$ are dense in the boundary of $U$, has full harmonic measure.

More generally, the statement remains true if instead of a finite degree of $f$ on $U$ we assume that the associated inner function $g = \varphi^{-1}\circ f \circ \varphi$, where $\varphi: \D \to U$ is a Riemann map, has non-singular Denjoy--Wolff point.
\end{ThmB}

The following example illustrates Theorem~B.

\begin{ex}[see \cite{baker-dominguez,faghen}]\label{ex:BD} Consider the map $f(z) = z + e^{-z}$, which is Newton's method applied to the entire function $F(z) = e^{-e^z}$. By \cite{bfjk}, Baker domains for Newton maps are simply connected. In fact, $f$ has infinitely many simply connected invariant Baker domains $U_k$, $k \in \Z$, such that $U_k = U_0 + 2k\pi i$, and $\deg N|_{U_k} = 2$. Since $f|_{U_k}$ is of doubly parabolic type (which can be easily checked using Theorem~\ref{thm:doubly}), it satisfies the condition of Theorem~B and hence the set of escaping points in the boundary of $U_k$ has zero harmonic measure. It seems plausible that all escaping points in $\bd U_k$ are non-accessible from $U_k$, while accessible repelling periodic points are dense in $\bd U_k$.
\end{ex}

Note that the assertion of Theorem~B does not hold for all $f$ of doubly parabolic type with infinite degree on $U$, as shown in the following example. 

\begin{ex}[{see \cite{aaronson3}, \cite[page 18]{doering-mane}}] \label{ex:mane}
Let $f:\C \to \clC$,
\[
f(z) = z - \sum_{n=0}^\infty \frac{2z}{z^2 - n^\delta}, \qquad 1 < \delta < 2.
\]

It is obvious that the upper half-plane $U = \{z: \Im(z) > 0\}$ is invariant under $f$. In \cite{aaronson3} it was shown that $\varrho_U(f^{n+1}(z), f^n(z)) \to 0$ as $n \to \infty$ for $z \in U$ and $f^n(x) \to \infty$ as $n \to \infty$ for almost every $x \in \R$ with respect to the Lebesgue measure. Hence, $U$ is a simply connected invariant Baker domain of doubly parabolic type and, since the harmonic and Lebesgue measure on $\bd U  = \R$ are mutually absolutely continuous (as the Riemann map $\varphi: \D \to U$ is M\"obius), the set of non-escaping points in $\bd U$ has zero harmonic measure. 
\end{ex}

In fact, there is a wider class of infinite degree Baker domains of doubly parabolic type, for which the assertion of Theorem~B still holds. Roughly speaking, the statement remains true when the hyperbolic distance between successive iterates of points in $U$ under $f$ tends to zero ``fast enough''. More precisely, we prove the following.

\begin{ThmC}
Let $f: \C \to \clC$ be a meromorphic map and let $U$ be a simply connected invariant Baker domain such that
\[
\varrho_U (f^{n+1}(z), f^n(z)) \leq \frac 1 n + O\left(\frac 1 {n^r} \right)
\]
as $n \to \infty$ for some $z \in U$ and $r > 1$. 
Then the set of points with forward trajectories dense in the boundary of $U$ $($in particular, non-escaping points$)$ has full harmonic measure.
\end{ThmC}

\begin{rem} By the Schwarz--Pick Lemma, the sequence $\varrho_U (f^{n+1}(z), f^n(z))$ is non-increasing. The condition in Theorem~C implies  $\varrho_U (f^{n+1}(z), f^n(z)) \to 0$, so $f|_U$ is of doubly parabolic type. 
For the Baker domain $U$ of the map $f$ described in Example~\ref{ex:mane} we have $c_1/n \leq \varrho_U (f^{n+1}(z), f^n(z)) \leq c_2/n$ for $z \in U$ and some $c_1, c_2 > 0$ (see \cite{doering-mane}), which shows that the estimate by $1/n + O(1/n^r)$ in Theorem~C cannot be changed to $c/n$ for arbitrary $c > 0$.

Note also that for an arbitrary Baker domain we have 
\[
\sum_{n= 0}^\infty \varrho_U (f^{n+1}(z), f^n(z)) = \infty
\]
for $z \in U$, since $f^n(z)$ converges to a boundary point of $U$ as $n \to \infty$. Therefore, the sequence $\varrho_U (f^{n+1}(z), f^n(z))$ cannot decrease to $0$ arbitrarily fast. 
\end{rem}

It is natural to ask whether there are actual examples of Baker domains satisfying the assumptions of Theorem~C. The following proposition answers this question in affirmative for a whole family of maps.

\begin{PropD} Let $f:\C \to \clC$ be a meromorphic map of
the form
\[
f(z) = z + a + h(z),
\]
where $a \in \C\setminus\{0\}$ and $h:\C \to \clC$ is a meromorphic map satisfying  
\[
|h(z)| < \frac{c_0}{(\Re(z/a))^r} \quad \text{for} \quad  \Re\left(\frac z
a\right) > c_1,
\]
for some constants $c_0, c_1 > 0$, $r > 1$. Then $f$ has an invariant Baker domain $U$ containing a half-plane $\{z
\in \C: \Re(z/a) > c\}$ for some $c \in\R$. Moreover, if $U$ is simply
connected $($e.g.~if $f$ is entire$)$, then $f$ on $U$ satisfies the
assumptions of Theorem~{\rm C} and, consequently, the set of points with dense  forward trajectories in the boundary of $U$ $($in particular, non-escaping points$)$ has full harmonic measure.
\end{PropD}

Using Proposition~D, one can find a number of examples
of Baker domains of doubly parabolic type and of infinite degree,
for which the set of points with dense forward trajectories in the boundary of $U$ has full harmonic measure. 
The first one is the classical example of a completely invariant Baker
domain, studied by Fatou in \cite{Fatou}.

\begin{ex}\label{ex:fatou}  Let $f:\C \to \clC$,
\[
f(z) = z + 1 + e^{-z}.
\]
Then $f$ obviously satisfies the conditions of Proposition~D
with $a = 1$.
\end{ex}

The next example was described in \cite{doering-mane} (see also
\cite{bfjkaccesses}).

\begin{ex}\label{ex:tan}
Let $f:\C \to \clC$,
\[
f(z) = z + \tan z.
\]
Then $U_+ = \{z: \Im(z) > 0\}$, $U_- = \{z: \Im(z) < 0\}$ are simply
connected invariant Baker domains, such that $\deg f|_{U_\pm} = \infty$.
Moreover, for large $\pm\Im(z)$,
\[
|f(z) - z \mp i| = |\tan z \mp i| \le \frac{2e^{\mp 2\Im(z)}}{1 - e^{\mp
2\Im(z)}} < 3 e^{\mp 2\Im(z)},
\]
so $f$ on $U_\pm$ satisfies the conditions of Proposition~D
with $a = \pm i$.
\end{ex}

The third example was described in \cite[Example 7.3]{bfjkaccesses}.

\begin{ex}\label{ex:absorb} Let $f:\C \to \clC$,
\[
f(z)=z+i + \tan z,
\]
which is Newton's method of the entire transcendental map $F(z)=\exp
\left(-\int_0^z \frac{du}{i+\tan u}\right)$. Then $f$ has a simply
connected invariant Baker domain $U$ containing an upper half-plane, such
that $\deg f|_U = \infty$. Since
\[
|f(z) - z -2i| = |\tan z - i|,
\]
the calculation in Example~\ref{ex:tan} shows that $f$ on $U$ satisfies
the conditions of Proposition~D with $a = 2i$. Note that $f$
has infinitely many simply connected invariant Baker domains $U_k$, $k \in
\Z$ of doubly parabolic type, such that $\deg f|_{U_k} = 2$, which satisfy the assumptions of Theorem~B.
\end{ex}

\begin{rem}[\bf The case of parabolic basins] \label{rem:parabolic} The results in Theorems~B and~C are valid  when instead of Baker domains we consider invariant parabolic basins, i.e. invariant Fatou components $U$ such that $f^n \to \zeta$ on $U$ as $n \to \infty$, with $\zeta \in \C$ being a boundary point of $U$ such that $f'(\zeta) = 1$. In this case instead of the escaping set $\mathcal{I}(f)$ one considers the set of points which converge to $\zeta$ under the iteration of $f$. 
Using extended Fatou coordinates, one can see that the dynamics in $U$ is semiconjugate to $z\mapsto z+1$ in $\C$, and hence every parabolic basin is of doubly parabolic type in the sense of Baker--Pommerenke--Cowen. In fact, in the case of parabolic basins of rational maps $f$, the results described in Theorem~B were proved in \cite{doering-mane} and \cite{adu}. 
\end{rem}

The paper is organized as follows. Section~2 contains some background and the statement of results we shall use in the proofs. Section~3 contains the proof of Theorem~A, Section~4 includes the proof of Theorem~C, while the proofs of Theorem~B and Proposition~D can be found in Section~5.

\ack{We wish to thank Phil Rippon and Gwyneth Stallard for inspiring discussions about their results on the boundary behaviour of maps on Baker domains. We thank Mariusz Urba\'nski for suggesting a strengthening of Theorem~B. We are grateful to the Institute of Mathematics of the Polish Academy of Sciences (IMPAN), Warsaw University, Warsaw University of Technology and Institut de Matem\`atiques de la Universitat de Barcelona (IMUB) for their hospitality while this paper was in progress.}

\section{Preliminaries} \label{sec:prelim}

\subsection*{Notation}

We denote, respectively, the closure and boundary of a set $A$ in $\C$ by $\overline{A}$ and $\bd A$. The Euclidean disc of radius $r$ centred at $z \in\C$ is denoted by $\D(z,r)$ and the unit disc $\D(0,1)$ is written as $\D$.
For a point $z \in \C$ and a set $A \subset \C$ we write
\[
\dist(z, A) = \inf_{w \in A} |z - w|.
\]
We denote by $\lambda$ the Lebesgue measure on the unit circle $\bdd$.  

We consider hyperbolic metric $\varrho_U$ on hyperbolic domains $U \subset \C$. In particular, we have
\begin{equation}\label{eq:hyp}
d\varrho_\D(z) = \frac{2 |dz|}{1 - |z|^2} \quad {\rm and} \quad \varrho_\D(z_1, z_2) = 2 \arctanh \left|\frac{z_1 - z_2}{1 - z_1 \overline z_2}\right|
\end{equation}
for $z,z_1,z_2 \in \D$. We use the standard estimate
\begin{equation}\label{eq:hypdist}
\varrho_U(z) \le \frac{2}{\dist(z, \bd U)}, \quad z \in U
\end{equation}
for hyperbolic domains $U \subset \C$ (see e.g.~\cite{carlesongamelin}). 

\subsection*{Boundary behaviour of holomorphic maps}
For a simply connected domain $U \subset \C$ we consider a Riemann map 
\[
\varphi : \D \to U.
\]
By Fatou's Theorem (see e.g.~\cite[Theorem 1.3]{pommerenke-book}), radial (or angular) limits of $\varphi$ exist at Lebesgue almost all points of $\bdd$. This does not prevent the existence of other curves approaching those boundary points such that their images have more complicated limiting behaviour and accumulate in a non-degenerate continuum. 

\begin{defn}[{\bf Ambiguous points}{}] \label{def:amb} Let $h: \D \to \C$. A point $\zeta \in \bdd$ is called \emph{ambiguous} for $h$, if there exist two curves $\gamma_1, \gamma_2: [0,1) \to \D$ landing at $\zeta$, such that the sets of limit points of the curves $h(\gamma_1)$, $h(\gamma_2)$ (when $t \to 1^-$) are disjoint.
\end{defn}

Note that for a normal holomorphic function $h$, one of the two sets of limit points must be a non-degenerate continuum or otherwise the two landing points would be the same. This is a consequence of Lehto--Virtanen Theorem (see e.g.~\cite{LVActa} or \cite[Section 4.1]{pommerenke-book}). 

The set of ambiguous points is small for any function $h$, as shown in the following theorem (see e.g.~\cite[Corollary 2.20]{pommerenke-book}).

\begin{thm}[{\bf Bagemihl ambiguous points theorem}{}]
\label{thm:amb}
An arbitrary function $h: \D \to \C$ has at most countably many ambiguous points.
\end{thm}

As in \cite{ripponstallard-escaping}, we use the following Pfl\"uger-type estimate on the boundary behaviour of conformal maps.

\begin{thm}[{\cite[Theorem 9.24]{pommerenke-book}}]\label{thm:Pfluger}
Let $\Phi: \D \to \C$ be a conformal map, let $V \subset \Phi(\D)$ be a non-empty open set and let $E$ be a Borel subset of $\bdd$. Suppose that there exist $\alpha \in (0, 1]$ and $\beta > 0$ such that:
\begin{itemize}
\item[\rm (a)] $\dist(\Phi(0), V) \ge \alpha |\Phi'(0)|$,
\item[\rm (b)] $\ell(\Phi(\gamma) \cap V) \ge \beta$ for every curve $\gamma \subset \D$ connecting $0$ to $E$,
\end{itemize}
where $\ell$ denotes the linear $($i.e.~$1$-dimensional Hausdorff$)$ measure in $\C$. Then
\[
\ell(E) \le 2 \pi \capacity (E) < \frac{15}{\sqrt{\alpha}} \; e^{-\frac{\pi \beta^2}{ \area V}},
\]
where $\capacity (E)$ denotes the logarithmic capacity of $E\subset \partial \D$.
\end{thm}

\subsection*{Inner functions and Baker--Pommerenke--Cowen classification}

Our goal is to study the dynamics of a meromorphic map $f$ restricted to  the boundary $\partial U$ of a simply connected invariant Baker domain $U$. We consider the pull-back $g: \D \to \D$ of $f$ under a Riemann map $\varphi: \D \to U$, i.e. the map
\[
g = \varphi^{-1}\circ f \circ \varphi.
\]
It is known that the map $g$ is an inner function.

\begin{defn}[{\bf Inner function}{}] \label{def:inner} A holomorphic map $h: \D \to \D$ is called an \emph{inner function}, if radial limits of $h$ have modulus $1$ at Lebesgue-almost all points of $\bdd$. 
\end{defn}

In this paper we deal with the harmonic measure on the boundary of $U$. 

\begin{defn}[{\bf Harmonic measure}{}] Let $U \subset \C$ be a simply connected domain and $\varphi : \D \to U$ let be a Riemann map. A \emph{harmonic measure} $\omega = \omega(U, \varphi)$ on $\bd U$ is the image under $\varphi$ of the normalized Lebesgue measure on the unit circle $\bdd$. 
\end{defn}

For more information on harmonic measure refer e.g.~to \cite{garnettmarshall}. 

If $h$ is an inner function then all its iterates $h^n$, $n = 1, 2,\ldots$, are also inner functions (see e.g. \cite{baker-dominguez}), which implies that the \emph{boundary map} on $\bdd$ (which we will denote by the same symbol $h$), defined by radial (or angular) limits of $h$, generates a dynamical system of iterations of $h$, defined Lebesgue-almost everywhere on $\bdd$.

\begin{defn}[{\bf Singular points}{}] \label{def:singular} A point $\zeta \in \bdd$ is singular for an inner function $h$, if $h$ cannot be extended holomorphically to any neighbourhood of $\zeta$. 
\end{defn}

Note that if an inner function $h$ has finite degree, then it is a finite Blaschke product, which extends to the Riemann sphere as a rational map. In this case, all points in $\bdd$ are non-singular for $h$. On the contrary, infinite degree inner functions must have at least one singular boundary point.

The asymptotic behaviour of the iterates of a holomorphic map $h$ in $\D$ is described by the classical Denjoy--Wolff Theorem (see e.g.~\cite[Theorem~3.1]{carlesongamelin}).

\begin{thm}[{\bf Denjoy--Wolff Theorem}{}]  
\label{thm:denjoy}
Let $h: \D \to \D$ be a holomorphic map, which is not identity nor an elliptic M\"obius transformation. Then there exists a point $p \in \overline{\D}$, called
the \emph{Denjoy--Wolff point} of $h$, such that the iterations $h^n$ tend to $p$ as $n \to \infty$ uniformly on compact subsets of $\D$.
\end{thm}

In this paper we deal with the case $p\in\bdd$. Then the dynamics of the map can be divided into three types, according to the \emph{Baker--Pommerenke--Cowen} classification of such maps.
The following result describes this classification, showing the existence of a semiconjugacy between the dynamics of $h$ and some M\"obius transformation (see \cite[Theorem 3.2]{cowen} and \cite[Lemma 1]{konig}).

\begin{thm}[{\bf Cowen's Theorem}{}]\label{thm:cowen}
Let $h: \D \to \D$ be a holomorphic map with the Denjoy--Wolff point $p \in \bdd$. Then there exists a simply connected domain $V \subset \D$, a domain 
$\Omega$ equal to $\HH = \{z\in\C: \Re(z) > 0\}$ or $\C$, a holomorphic map $\psi: \D \to \Omega$, 
and a M\"obius transformation $T$ mapping $\Omega$ onto itself, such that:
\begin{itemize}
\item[$(a)$] $V$ is absorbing in $\D$ for $h$, i.e.~$h(V) \subset V$ and for every compact set $K \subset \D$ there exists $n \ge 0$ with $h^n(K) \subset V$,
\item[$(b)$] $\psi(V)$ is absorbing in $\Omega$ for $T$,
\item[$(c)$] $\psi \circ h = T \circ \varphi$ on $\D$,
\item[$(d)$] $\psi$ is univalent on $V$.
\end{itemize}
Moreover, up to a conjugation of $T$ by a M\"obius transformation preserving $\Omega$, one of the following three cases holds:
\begin{align*}
&\Omega = \HH,\; T(\omega) = a\omega \;\text{ for some } a > 1 &&(\text{hyperbolic 
type}),\\
&\Omega = \HH,\; T(\omega) = \omega \pm i &&(\text{simply parabolic type}),\\
&\Omega = \C, \; T(\omega) = \omega + 1 &&(\text{doubly parabolic type}).
\end{align*}
\end{thm}

In view of this theorem, we say that $f|_U$ is of hyperbolic, simply parabolic or doubly parabolic type if the same holds for the associated inner function $g = \varphi^{-1}\circ f \circ \varphi$. Generally, it is not obvious how to determine the type of Baker domain just by looking at the dynamical plane, since it can depend on the dynamical properties of $f$ and geometry of the domain $U$.

If an inner function extends to a neighbourhood of the Denjoy--Wolff point $p \in \bdd$, its type can be determined by the local behaviour of its trajectories near $p$ (see \cite{bergweiler-sing,faghen}). More precisely, we have the following.

\begin{thm}[{\cite[Theorem 2]{bergweiler-sing}}]\label{thm:traj}
Let $h: \D \to \D$ be an inner function with the non-singular Denjoy--Wolff point $p \in \bdd$. Then the following hold.
\begin{itemize}
\item[\rm (a)] If $h$ is of hyperbolic type, then $p$ is an attracting fixed point of $h$ with $h'(p) \in (0, 1)$.

\item[\rm (b)] If $h$ is of simply parabolic type, then $p$ is a parabolic fixed point of $h$ of multiplicity $2$, i.e.~$h(z) = p + (z - p) + \alpha (z - p)^2 + O((z-p)^3)$ as $z \to p$ for some $\alpha \neq 0$. Moreover, 
\[
\varrho_\D(h^{n+1}(z), h^n(z)) = b + \frac{c}{n^3} + O\left(\frac{1}{n^4}\right)
\]
for $z \in \D$ as $n \to \infty$, where $b = \lim_{n\to \infty}\varrho_\D(h^{n+1}(z), h^n(z)) >0$ and $c \ge 0$.

\item[\rm (c)] If $h$ is of doubly parabolic type, then $p$ is a parabolic fixed point of $h$ of multiplicity $3$, i.e.~$h(z) = p + (z - p) + \beta (z - p)^3 + O((z-p)^4)$ as $z \to p$ for some $\beta \neq 0$. Moreover, 
\[
\varrho_\D(h^{n+1}(z), h^n(z)) = \frac{1}{2n} + \frac{d \ln n}{n^2} + O\left(\frac{1}{n^2}\right)
\]
for $z \in \D$ as $n \to \infty$, where $d\in \R$.
\end{itemize}

\end{thm}

An easy consequence of this theorem is the following corollary.

\begin{cor}\label{cor:traj}
Let $h: \D \to \D$ be an inner function of finite degree larger than $1$ with the Denjoy--Wolff point $p \in \bdd$. Denote by the same symbol the holomorphic extension of $h$ to $\clC$. Then:
\begin{itemize}

\item[\rm (a)] If $h|_\D$ is of hyperbolic type, then the attracting basin of $p$ is connected, non-simply connected and contains $\D \cup (\clC \setminus \overline\D)$, so $\calj(h) \subsetneq \bdd$.

\item[\rm (b)]
If $h|_\D$ is of simply parabolic type, then the parabolic basin of $p$ consists of one non-simply connected attracting petal, which contains $\D \cup (\clC \setminus \overline\D)$, so $\calj(h) \subsetneq \bdd$.

\item[\rm (c)]
If $h|_\D$ is of doubly parabolic type, then the parabolic basin of $p$ consists of two simply connected attracting petals $\D$ and $\clC \setminus \overline\D$, so $\calj(h) = \bdd$.

\end{itemize}

\end{cor}

The characterization of doubly parabolic type can be extended to the general case in the following way.
 
\begin{thm}[{\cite[Theorem~A]{absorb}}, \cite{konig}]\label{thm:doubly}
Let $f: \C \to \clC$ be a meromorphic map with a simply connected invariant Baker domain $U$. Then the following conditions are equivalent.
\begin{itemize}
\item[\rm (a)] $f|_U$ is of doubly parabolic type.
\item[\rm (b)] $\varrho_U(f^{n+1}(z), f^n(z)) \to 0$ as $n \to \infty$ for some $($every$)$ $z \in U$.
\item[\rm (c)] $|f^{n+1}(z)- f^n(z)|/\dist(f^n(z), \bd U) \to
0$ as $n \to \infty$ for some $($every$)$ $z \in U$.
\end{itemize}
\end{thm}

\subsection*{Ergodic theory of inner functions}

First, we recall some basic notions used in abstract ergodic theory (for more details, refer e.g. to \cite{aaronson-book,petersen-book}).

\begin{defn}\label{defn:erg} Let $\mu$ be a measure on a space $X$ and let $T: X \to X$ be a $\mu$-measurable transformation. Then we say that $T$ (or $\mu$) is:
\begin{itemize}
\item[$\circ$] \emph{exact}, if for every measurable $E \subset X$, such that for every $n$ we have $E = T^{-n}(X_n)$ for some measurable $X_n \subset X$, there holds $\mu(E) = 0$ or $\mu(X \setminus E) = 0$,
\item[$\circ$] \emph{ergodic}, if for every measurable $E \subset X$ with $E = T^{-1}(E)$ there holds $\mu(E) = 0$ or $\mu(X \setminus E) = 0$,
\item[$\circ$] \emph{recurrent}, if for every measurable $E \subset X$ and $\mu$-almost every $x \in E$ there exists an infinite sequence of positive integers $n_k \to \infty$, $k = 1, 2, \ldots$, such that $T^{n_k}(x) \in E$. 
\item[$\circ$] \emph{conservative}, if for every measurable $E \subset X$ of positive measure $\mu$ there exist $n, m \ge 0$, $n \neq m$, such that $T^{-n}(E) \cap T^{-m}(E) \neq \emptyset$ $($i.e.~$T$ has no \emph{wandering} sets of positive measure),
\item[$\circ$] \emph{preserving} $\mu$ (in this case we say that $\mu$ is \emph{invariant}), if for every measurable $E \subset X$ we have $\mu(T^{-1}(E)) = \mu(E)$,
\item[$\circ$] \emph{non-singular}, if for every measurable $E \subset X$ we have $\mu(T^{-1}(E)) = 0$ if and only if $\mu(E) = 0$.
\end{itemize}
\end{defn}

Obviously, exactness implies ergodicity and invariance implies non-singularity. Moreover, there holds

\begin{thm}[\cite{halmos}]\label{thm:halmos}
The measure $\mu$ is conservative if and only if it is recurrent.
\end{thm}

If $\mu$ is finite and invariant, then the Poincar\'e Recurrence Theorem (see e.g.~\cite{petersen-book}) asserts that $T$ is recurrent. Note that this does not extend to the case of infinite invariant measures. On the other hand, the following holds (see e.g.~\cite{aaronson-book}).

\begin{thm}\label{thm:conserv}
If $\mu$ is non-singular, then the following are equivalent:
\begin{itemize}
\item[\rm (a)] $T$ is conservative and ergodic.
\item[\rm (b)] For every measurable $E \subset X$ of positive measure $\mu$,  for $\mu$-almost every $x \in X$ there exists an infinite sequence of positive integers $n_k \to \infty$, $k = 1, 2, \ldots$, such that $T^{n_k}(x) \in E$.  
\end{itemize}

\end{thm}

Recall that an inner function $h: \D \to \D$ generates a dynamical system of iterations of $h$ on $\bdd$, defined Lebesgue-almost everywhere on $\bdd$. 
We will use the following fundamental dichotomy, proved by Aaronson \cite{aaronson2} (see also \cite[Theorems~4.1 and 4.2]{doering-mane}) on the boundary behaviour of inner functions.

\begin{thm}[\cite{aaronson2}] \label{thm:ifthen}
Let $h: \D \to \D$ be an inner function. Then the following hold.
\begin{itemize}
\item[\rm (a)] If $\sum_{n=1}^\infty (1 - |h^n(z)|) < \infty$ for some $z \in \D$, then $h^n$ converges to a point $p \in\bdd$ almost everywhere on $\bdd$. 
\item[\rm (b)] If $\sum_{n=1}^\infty (1 - |h^n(z)|) = \infty$ for some $z \in \D$, then $h$ on $\bdd$ is conservative with respect to the Lebesgue measure. 
\end{itemize}
\end{thm}

In \cite{aaronson3} (see also \cite[Theorem~3.1]{doering-mane}), the following characterization of the exactness of $h$ was established.

\begin{thm}[\cite{aaronson3}]\label{thm:exact}
Let $h: \D \to \D$ be an inner function with the Denjoy--Wolff point in $\bdd$. Then $h$ on $\bdd$ is exact with respect to the Lebesgue measure if and only if $h$ is of doubly parabolic type. 
\end{thm}

If the Denjoy--Wolff point of an inner function $h$ is in $\D$, then $h$ on $\bdd$ preserves an absolutely continuous finite (harmonic) measure, which is exact (see e.g.~\cite{doering-mane}). Suppose now $h$ has the Denjoy--Wolff point $p\in\bdd$. Then $h$ has no longer an absolutely continuous finite invariant measure. However, in the parabolic case it preserves a $\sigma$-finite absolutely continuous measure. More precisely, define a measure $\mu_p$ on $\bdd$ by 
\begin{equation}\label{eq:mu}
\mu_p(E) = \int_E \frac{d\lambda(w)}{|w-p|^2}
\end{equation}
for Lebesgue-measurable sets $E \subset \bdd$. 
A short calculation shows that $\mu_p$ is equal (up to a multiplication by a constant) to the image of the Lebesgue measure on $\R$ under a M\"obius transformation $M$ mapping conformally the upper half-plane onto $\D$ with $M(\infty) = p$. 
It is obvious that the Lebesgue measure $\lambda$ on $\bdd$ and the measure $\mu_p$ are mutually absolutely continuous, i.e.
\begin{equation}\label{eq:zero}
\lambda(E) = 0 \Longleftrightarrow\mu_p(E) = 0
\end{equation}
for Lebesgue-measurable sets $E \subset \bdd$.

It is known (see e.g.~\cite{pommer}) that an inner function $h$ with the Denjoy--Wolff point $p \in \bdd$ has an angular derivative equal to some $q \in (0,1]$, where the case $q < 1$ corresponds to hyperbolic type of $h$, while $q = 1$ corresponds to (simply or doubly) parabolic type. We have
\[
q = \lim_{z \to p} \frac{1 - |h(z)|}{1 - |z|}
\]
in the sense of angular limit and 
\[
q = \lim_{n\to\infty}(1 - |h^n(z)|)^{\frac{1}{n}} \qquad \text{for every } z \in \D.
\]
The following result asserts in particular that in the parabolic case the measure $\mu_p$ is invariant.

\begin{thm}[{\cite[Theorem 4.4]{doering-mane}}]\label{thm:mu} Let $h: \D \to \D$ be an inner function with the Denjoy--Wolff point $p \in \bdd$. Then 
\[
\mu_p(h^{-1}(E)) = q\mu_p(E)
\]
for every Lebesgue-measurable set $E \subset \bdd$. 
\end{thm}

\section{Proof of Theorem~A} \label{sec:proofA}

Throughout this section we assume that $f: \C \to \clC$ is a meromorphic map with a simply connected invariant Baker domain $U$, such that $f|_U$ is of hyperbolic or simply parabolic type, and the associated inner function 
\[
g = \varphi^{-1}\circ f \circ \varphi,
\]
where $\varphi: \D \to U$ is a Riemann map, has the non-singular Denjoy--Wolff point $p \in \bdd$. As mentioned in Section~\ref{sec:prelim}, this includes the case when $f$ has finite degree on $U$. 

The proof of Theorem~A extends the arguments used by Rippon and Stallard for univalent Baker domains. As in \cite{ripponstallard-escaping}, we use the Pfl\"uger-type estimate on the boundary behaviour of conformal maps included in Theorem~\ref{thm:Pfluger}. The following lemma makes a crucial step in the proof of Theorem A.

\begin{lem}\label{lem:main} The following statements hold.
\begin{itemize}
\item[\rm(a)] There exist an arbitrarily small neighbourhood $W$ of the Denjoy--Wolff point $p$ of $g$ and a point $w_0 \in \bdd \cap W \setminus \{p\}$, such that $g^n(w_0) \in W$ for every $n \ge 0$. 
\item[\rm(b)] Let $J_n \subset \bdd$ be the closed arc connecting $g^n(w_0)$ with $g^{n+1}(w_0)$ in $W$ and let
\[
B_n = B_n(w_0, M) = \{z \in J_n: |\varphi(z)| < M\}
\]
$($in the sense of radial limits of $\varphi)$ for $M > 0$. Then 
\[
\sum_{n=0}^\infty \frac{\lambda(B_n)}{a^n} < \infty,
\]
where $a = g'(p)$.
\end{itemize}
\end{lem}

\begin{rem}\label{rem:cap} Recall that by Beurling's Theorem (see \cite{Beu40}), the map $\varphi$ extends continuously to $\bdd$ (in the sense of radial limits) up to a set of logarithmic capacity $0$. Repeating the arguments used in the proof of \cite[Theorem~3.1]{ripponstallard-escaping}, one can show that in fact
\[
\sum_{n=0}^\infty \frac{1}{\ln(a^n/\capacity B_n)} < \infty,
\]
where $\capacity$ denotes logarithmic capacity
$($with the convention $\frac{1}{\ln(a^n/\capacity B_n)} = 0$ if $\capacity B_n = 0)$. 
\end{rem}

\begin{proof}[Proof of Lemma~\rm\ref{lem:main}]

By assumption, $g$ extends holomorphically to a neighbourhood of $p$. As $g^n \to p$ on $\D$, by continuity we have $g(p) = p$, and since $g$ is an inner function, $p$ is an attracting or parabolic point of $g$ with $a = g'(p) \in (0,1)$ or $a = 1$. In fact, these two possibilities correspond to the cases when $f|_U$ is, respectively, hyperbolic or simply parabolic (see Theorem~\ref{thm:traj}). The proof of the lemma splits into two parts dealing with these cases.

\medskip
\noindent\emph{Case $1$$:$ $f|_U$ is hyperbolic}
\medskip

In this case the Denjoy--Wolff point $p$ of $g$ is an attracting fixed point of $g$ in $\bdd$ with $a = g'(p) \in (0,1)$. Let $\Psi$ be a conformal map from a neighbourhood of $p$ to a neighbourhood of $0$ conjugating $g$ to $z \mapsto az$, i.e. $\Psi(p)=0$ and 
\begin{equation}\label{eq:conj}
\Psi(g(z)) = a\Psi(z)
\end{equation}
for $z$ near $p$. Taking $W=\Psi^{-1}\left(\D(0,\varepsilon)\right)$ for $\varepsilon>0$ small enough, we have $\overline{g(W)}\subset W$. In particular, $g^n$ is defined in $W$ for all $n\geq 0$ and for every $w_0 \in \bdd \cap W \setminus \{p\}$, we have $g^n(w_0) \in W$ for $n \ge 0$. 
For further purposes, we choose the point $w_0$ such that $a^{-1} |\Psi(w_0)| \subset \D(0,\varepsilon)$. This ends the proof of the statement~(a).

Now we prove the statement~(b). By definition, $J_n \subset \bdd$ is the closed arc connecting $g^n(w_0)$ and $g^{n+1}(w_0)$ in $W$. Thus, by construction, $\bigcup_{n\ge 0}J_n \subset W$ and $g^n \to p$ on $\bigcup_{n\ge 0}J_n$. 

\begin{figure}[hbt!]
\centering
\def\svgwidth{0.95\textwidth}
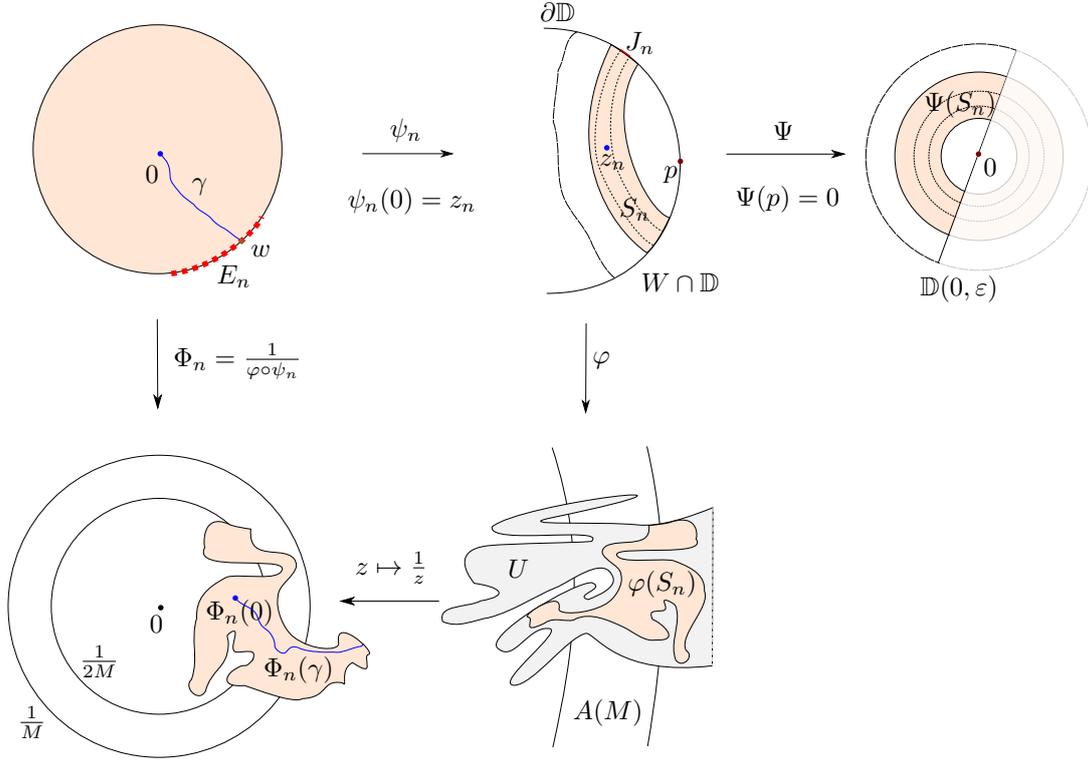
\caption{Sketch of the proof of Lemma~\ref{lem:main}.} 
\label{fig:One_Two_Accesses}
\end{figure}

Using the map $\Psi$, we define domains $S_n \subset \D$, $n\geq 0$, as 
\[
S_n = \Psi^{-1}(\{z \in \C: a^{n+2} |\Psi(w_0)| < |z| < a^{n-1} |\Psi(w_0)|\}) \cap \D.
\]
Since $a^{-1} |\Psi(w_0)| \subset \D(0,\varepsilon)$ and $\Psi$ is conformal, for sufficiently small $\varepsilon$ the domains $S_n$ are non-empty, open and simply connected in $\D$, with a Jordan boundary consisting of four analytic curves, one of them strictly containing $J_n$. Note that for $n \ge 1$, the domain $S_n$ intersects only with $S_{n-1}$ and $S_{n+1}$.

Choose a point $z_0 \in S_0$ and let $z_n = g^n(z_0) \in S_n$. Define 
\[
\psi_0: \D \to S_0
\]
to be a Riemann map with $\psi_0(0) = z_0$ and let
\[
\psi_n(z) = \Psi^{-1}(a^n(\Psi(\psi_0(z)))), \qquad z \in \D.
\]
By the definition of $S_n$, the map $\psi_n$ is a Riemann map from $\D$ onto $S_n$ with $\psi_n(0) = z_n$. Note that since $J_n$ is strictly contained in an open circle arc in $\bd S_n$, the map $\psi_n$ extends analytically (by the Schwarz reflection) to a neighbourhood of each point of $\psi_n^{-1}\left(J_n\right)$. 

Fix $M>0$ and for large $n\geq 0$ define the following three sets:
\[
A(M) = \{z \in \C: M < |z| < 2M\}, \ \  V_n = \{1/z: z \in \varphi(S_n) \cap A(M)\},  \ \   
E_n = \psi_n^{-1}(B_n \setminus \A(\varphi)),
\]
where $\A(\varphi)$ is the set of ambiguous points of $\varphi$ in $\bdd$ (see Definition~\ref{def:amb}). Note that $V_n$ is contained in the annulus $\{1/(2M) < |z| < 1/M\}$.

Now we show
\begin{equation}\label{eq:series}
\sum_{n=n_0}^\infty \lambda(E_n) < \infty
\end{equation}
for some $n_0$. 
First, note that if the set $V_n$ is empty, then $\lambda(E_n) = 0$. To see this, note that if $V_n = \emptyset$, then $\varphi(S_n) \subset \overline{\D(0, M)} \cup (\C \setminus \D(0, 2M))$. Since $\varphi(z_n) \in \varphi(S_n)$ and $\varphi(z_n) \to \infty$ as $n \to \infty$, we can assume $|\varphi(z_n)| > 2M$. As $\varphi(S_n)$ is connected, this implies $\varphi(S_n) \subset \C \setminus \D(0, 2M)$, which gives $B_n = E_n = \emptyset$, in particular $\lambda(E_n)=0$. Hence, we can assume that $V_n$ is not empty for large $n$. 

Since $U \neq \C$, by a conformal change of coordinates, we may assume $0 \notin U$, in particular $0 \notin \varphi(\psi_n(\D))$ for $n \ge 0$. 
Consider the mapping
\[
\Phi_n(z) = \frac{1}{\varphi(\psi_n(z))}, \qquad z \in \D. 
\]
Then $\Phi_n$ is a conformal map from $\D$ into $\C$ and $V_n$ is an open set in $\Phi_n(\D)$. We will apply Theorem~\ref{thm:Pfluger} to $\Phi=\Phi_n$, $V=V_n$, $E=E_n$. To check the assumption~(a) of this theorem, note that by the One-quarter Koebe Theorem, 
\[
\Phi_n(\D) \supset \D(\Phi_n(0), |\Phi_n'(0)|/4).
\]
Note also that for $n$ large enough we may assume $|\Phi_n(0)|=1/|\varphi(z_n)| <1/2M$ since, in fact, we have $1/|\varphi(z_n)| \to 0$ as $n \to \infty$.

As $z_0 \notin S_n$ for $n > 2$, we have $1/\varphi(z_0) \notin \Phi_n(\D)$ for large $n$, which implies that
\begin{equation} \label{eq:koebe}
|\Phi_n'(0)| < 4\left|\frac{1}{\varphi(z_0)} - \Phi_n(0)\right| < \frac{5}{|\varphi(z_0)|}.
\end{equation}
On the other hand, $\emptyset \neq V_n \subset \{z \in\C: |z| \ge 1/(2M)\}$, so
\[
\dist(\Phi_n(0), V_n) \ge \frac 1 {2M} - |\Phi_n(0)| > \frac 1 {3M} \ge \frac{|\varphi(z_0)|}{15M}|\Phi_n'(0)|,
\]
where the last inequality follows from  (\ref{eq:koebe}).
Thus the assumption~(a) of Theorem~\ref{thm:Pfluger} is satisfied with $\alpha = |\varphi(z_0)|/(15M)$.

To check the assumption~(b) of Theorem~\ref{thm:Pfluger}, take
a curve $\gamma \subset \D$ connecting $0$ to a point $w \in E_n$. 
Note that by the definitions of $B_n$ and $E_n$, the radial limit of $\Phi_n$ at $w$ exists and has modulus larger than $1/M$. Moreover, $w$ is not an ambiguous point of $\Phi_n$, so the limit set of $\Phi_n(\gamma)$ contains the radial limit of $\Phi_n$ at $w$. Hence, there is a sequence of points  $w_k \in \gamma$ converging to $w$, such that 
$|\Phi_n(w_k)| > 1/M$.
Since we know that $|\Phi_n(0)| < 1/(2M)$,  it follows that the curve $\Phi_n(\gamma)$ must joint two components of the complement of the annulus $\{1/(2M) < |z| < 2M\}$, which implies $\ell(\Phi_n(\gamma) \cap V_n) \ge 1/(2M)$. Hence, the assumption~(b) of Theorem~\ref{thm:Pfluger} is satisfied with $\beta = 1/(2M)$.

Now Theorem~\ref{thm:Pfluger} applied to $\Phi_n, V_n$ and $E_n$ for $n$ large enough gives 
\begin{equation}\label{eq:area}
\lambda(E_n) = \ell(E_n) < c_1 e^{-\frac{c_2}{\area V_n}}
\end{equation}
for some $c_1, c_2 > 0$ independent of $n$. Since by definition, $V_n \subset \D(0, 1/M)$ and $V_n$ can intersect only with $V_{n-1}$ and $V_{n+1}$, we have
\[
\sum_{n = 0}^\infty \area V_n \le \frac{3\pi}{M^2},
\]
in particular $\area V_n \to 0$ as $n \to \infty$. Hence, $e^{-\frac{c_2}{\area V_n}} < c_3 \area V_n$ for some $c_3 > 0$ independent of $n$, so \eqref{eq:area} gives
\[
\sum_{n = n_0}^\infty \lambda(E_n) < c_1c_3 \sum_{n = n_0}^\infty\area V_n < \infty
\]
for some $n_0$, which ends the proof of \eqref{eq:series}.

Since $\psi_0$ is holomorphic in a neighbourhood of $\psi_n^{-1}\left(J_0\right)$, it is bi-Lipschitz on $\psi_0^{-1}\left(J_0\right)$. Similarly, $\Psi$ is bi-Lipschitz on $W$. Hence, by the definition of $\psi_n$, for $n$ large enough we have
\begin{equation}\label{eq:lambda}
\lambda(B_n \setminus \A(\varphi)) = \lambda(\psi_n(E_n)) \le c_4 a^n \lambda(E_n)
\end{equation}
for some $c_4 > 0$ independent of $n$. 
Moreover, the set $\A(\varphi)$ of ambiguous points of $\varphi$ is at most countable (see Theorem~\ref{thm:amb}), so 
\[
\lambda(B_n) = \lambda(B_n \setminus \A(\varphi)).
\]
This together with \eqref{eq:lambda} gives 
\[
\frac{\lambda(B_n)}{a^n} \le c_4 \lambda(E_n) 
\]
so by  \eqref{eq:series},
\[
\sum_{n=0}^\infty \frac{\lambda(B_n)}{a^n} < \infty,
\]
which ends the proof in case~1.

\bigskip

\noindent
\noindent\emph{Case $2$$:$ $f|_U$ is simply parabolic}
\medskip

In this case the Denjoy--Wolff point $p$ of $g$ is a parabolic fixed point of $g$ in $\bdd$ with $a = g'(p) = 1$, of multiplicity $2$ (see Theorem~\ref{thm:traj}). By the local analysis of $g$ near such point (see e.g.~\cite{carlesongamelin}) and the fact that $g$ preserves $\bdd$ near $p$, there is an open arc $J \subset \bdd$ containing $p$, such that $J \setminus \{p\} = J^+ \cup J^-$, $g(J^-) \subset J^-$ and $g^n \to p$ on $J^-$, while $g(J^+) \supset J^+$ and points of $J^+$ escape from $J^+$ under iteration of $g$. Moreover, there is a conformal map $\Psi$ (Fatou coordinates) defined on an open region containing $J^-$, which conjugates $g$ to $z \mapsto z + 1$, i.e. 
\begin{equation}\label{eq:conj2}
\Psi(g(z)) = \Psi(z) + 1
\end{equation}
for $z$ in the domain of definition of $\Psi$. For a precise definition and properties of the map $\Psi$ see e.g.~\cite{carlesongamelin,milnor}. In particular, any neighbourhood $W$ of $p$ contains the set $\Psi^{-1}(\{\Re(z) > R\})$ for $R\in \mathbb R^+$ large enough and $J^- \cap \Psi^{-1}(\{\Re(z) > R\}) \neq \emptyset$. Hence, we can choose $w_0 \in J^- \cap \Psi^{-1}(\{\Re(z) > R\}) \subset W$, such that $\Re \left(\Psi(w_0)\right) > R + 1$. By \eqref{eq:conj2}, we have $g(\Psi^{-1}(\{\Re(z) > R\})\subset \Psi^{-1}(\{\Re(z) > R\})$, so $g^n(w_0) \in W$ for $n\ge 0$. Moreover, if $J_n \subset \bdd$ is the closed arc connecting $g^n(w_0)$ and $g^{n+1}(w_0)$ in $W$, then $\bigcup_{n\ge 0}J_n \subset J^-$, $\Psi$ is defined on $\bigcup_{n\ge 0}J_n$ and $g^n \to p$ on $\bigcup_{n\ge 0}J_n$. Moreover, $\Psi$ is bi-Lipschitz on $\bigcup_{n\ge 0}J_n$ (see e.g.~\cite{carlesongamelin,milnor}). 

For $n \ge 0$ let
\[
S_n = \Psi^{-1}(Q_n) \cap \D,
\]
where
\begin{multline*}
Q_n = \{z \in \C: \Re(z) \in \left(\Re(\Psi(w_0)) + n - 1, \Re(\Psi(w_0)) + n + 2\right),\; \\
\Im(z) \in \left(\Im(\Psi(w_0)) - 1, \Im(\Psi(w_0)) + 1\right)\}.
\end{multline*}

Then $S_n$ is a simply connected region in $\D$, with a Jordan boundary consisting of four analytic curves, one of them strictly containing $J_n$. Like previously, choose a point $z_0 \in S_0$, let $z_n = g^n(z_0) \in S_n$ and
define
\[
\psi_0: \D \to S_0
\]
to be a Riemann map such that $\psi_0(0) = z_0$. Set
\[
\psi_n(z) = \Psi^{-1}(\Psi(\psi_0(z))+n), \qquad z \in \D.
\]
Then $\psi_n$ is a Riemann map from $\D$ onto $S_n$ such that $\psi_n(0) = z_n$. Now we can proceed with the rest of the proof in the same way as in the hyperbolic case.
\end{proof}

\begin{lem}\label{lem:local}
The following statements hold.
\begin{itemize}
\item[\rm a)] There exists an open arc $I \subset \bdd$, with $p\in \overline{I}$, such that 
\[
\bigcup_{n = 0}^\infty g^{-n}(I) \cup \{p\} = \{z \in \bdd: g^n(z) \to p \text{ as }n \to \infty\}.
\]
\item[\rm (b)] $f^n(\varphi(z)) \to \infty$ as $n \to \infty$ for Lebesgue-almost all points $z \in I$.
\end{itemize}

\end{lem}

\begin{rem} Using Remark~\ref{rem:cap} and repeating the arguments from \cite{ripponstallard-escaping} one can show that in fact $f^n(\varphi(z)) \to \infty$ for all points $z \in I$ except of a set of logarithmic capacity zero.
\end{rem}

\begin{proof}[Proof of Lemma~\rm\ref{lem:local}]

We use the notation from Lemma~\ref{lem:main} and its proof. In the simply parabolic case, when the Denjoy--Wolff point $p$ is a parabolic fixed point with one attracting petal, let
\[
I = \bigcup_{n = 0}^\infty J_n.
\]
In the hyperbolic case, when $p$ is an attracting fixed point, we define
\[
I = \bigcup_{n = 0}^\infty J_n \cup J_n^{\prime} \cup\{p\},
\]
where $J_n$ and $J_n'$ are the arcs defined in Lemma~\ref{lem:main} for, respectively, two points $w_0$ and $w_0'$ situated in $\bdd$ on both sides of $p$.  

The first assertion of the lemma, in both cases, follows directly from the definition of $I$. Indeed, in the hyperbolic case $I$ contains on open arc in $\bdd$ containing $p$, while in the simply parabolic case it forms a one-sided neighbourhood of $p$ in $\bdd$, contained in the unique attracting petal of $p$.

Now we prove the second assertion. Recall that by Fatou's Theorem, the radial limit of $\varphi$ exists at Lebesgue-almost all points of $\bdd$. Moreover, $g$ is bi-Lipschitz in a neighbourhood of $p$ (and hence preserves zero measure sets). This implies that the radial limits of $\varphi \circ g^n$, $n \ge 0$ exist at almost all points of $\bdd$. Hence, to prove the second assertion of the lemma, it is sufficient to show that the Lebesgue measure of the set $Y$ of points in $I$, for which the radial limits of $\varphi \circ g^n$ exist and do not tend to $\infty$ for $n \to \infty$, is equal to zero. In the simply parabolic case, the set $Y$ can be written as
\[
\bigcup_{n = 0}^\infty \bigcup_{M \in \N} \bigcap_{m \in \N} \bigcup_{k \ge m} \{z \in J_n: |\varphi(g^k(z))| < M\} = \bigcup_{n = 0}^\infty \bigcup_{M \in \N} \bigcap_{m \in \N} \bigcup_{k \ge m} g^{-k}(B_{n+k}(w_0, M)).
\]
Similarly, in the hyperbolic case $Y$ is equal to
\[
\bigcup_{n = 0}^\infty \bigcup_{M \in \N} \bigcap_{m \in \N} \bigcup_{k \ge m} g^{-k}(B_{n+k}(w_0, M)) \cup g^{-k}(B_{n+k}(w_0', M)) \cup \{p\},
\]
where $w_0$ and $w_0'$ are situated in $\bdd$ on both sides of $p$.  
Hence, to prove that $Y$ has Lebesgue measure zero, it is enough to show that
\[
\lim_{m \to \infty} \lambda \left(\bigcup_{k \ge m} g^{-k}(B_{n+k}(w_0, M))\right) = 0
\]
for every $M, n$ (the case of $w_0'$ is analogous). But
\[
\lim_{m \to \infty} \lambda \left(\bigcup_{k \ge m} g^{-k}(B_{n+k}(w_0, M))\right) \leq
\lim_{m\to \infty} \sum_{k=m}^{\infty} \lambda (g^{-k}(B_{n+k}(w_0, M))),
\]
so it is sufficient to show 
\begin{equation}\label{eq:toshow}
\sum_{k = 0}^\infty \lambda(g^{-k}(B_{n+k}(w_0, M))) < \infty.
\end{equation}
To do it, observe that because of \eqref{eq:conj}, \eqref{eq:conj2} and the fact that $\Psi$ is bi-Lipschitz on $\bigcup_{n\ge 0}J_n$ in both (hyperbolic and simply parabolic) cases, we have
\[
c_1 < \frac{|(g^k)'(z)|}{a^k} < c_2
\]
for every $z \in J_n$ and some $c_1, c_2 >0$ independent of $n,k$, where $g'(p)=a \in (0, 1]$ (see the proof of Lemma~\ref{lem:main}). This implies
\[
\lambda(g^{-k}(B_{n+k}(w_0, M))) \le \frac{c_3}{a^{k}} \lambda(B_{n+k}(w_0, M)) \le c_3 \frac{\lambda(B_{n+k}(w_0, M))}{a^{n+k}}
\]
for some $c_3>0$. The latter inequality together with Lemma~\ref{lem:main} shows \eqref{eq:toshow}, which implies that $Y$ has Lebesgue measure $0$ and ends the proof of the lemma.

\end{proof}

\begin{proof}[Proof of Theorem~\rm A]
By Lemma~\ref{lem:local}, $g^n \to p$ on $I$, where $I \subset \bdd$ has positive Lebesgue measure, so the map $g$ on $\bdd$ is not recurrent with respect to the Lebesgue measure. Therefore, by Theorem~\ref{thm:ifthen}, $g^n \to p$ Lebesgue-almost everywhere on $\bdd$. Thus, using again  Lemma~\ref{lem:local}, we obtain that for Lebesgue-almost every point $z \in \bdd$ there exists $k \ge 0$ such that $g^k(z) \in I$, and  $f^n(\varphi(g^k(z))) \to \infty$ as $n \to \infty$. Since $f\circ\varphi = \varphi\circ g$, we conclude that  
$f^n(\varphi(z)) \to \infty$ as $n \to \infty$ for Lebesgue-almost every point $z \in \bdd$, which is equivalent to say that almost every point in $\bd U$ with respect to harmonic measure escapes to infinity under iteration of $f$. 

\end{proof}

\section{Proof of Theorem~C} \label{sec:proofB}

In this section, devoted to the proof of Theorem~C, we assume that $U$ is a simply connected invariant Baker domain of a meromorphic map $f:\C\to \chat$, such that there exist $z \in U$, $r > 1$ and $c > 0$ such that
\[
\varrho_U(f^{n+1}(z), f^n(z)) \le \frac 1 n + \frac c {n^r}
\]
for every $n \ge 1$ (it is obvious that we can assume $r < 2$). The assumption
implies in particular
\[
\varrho_U(f^{n+1}(z), f^n(z)) \to 0
\]
as $n \to \infty$, so by Theorem~\ref{thm:doubly}, the map $f|_U$ is of doubly parabolic type.
As previously, we consider the associated inner function 
\[
g = \varphi^{-1}\circ f \circ \varphi,
\]
where $\varphi: \D \to U$ is a Riemann map. We will show that the series $\sum_{n=1}^\infty (1-|g^n(w)|)$ is divergent for $w \in \D$ and then apply Theorem~\ref{thm:ifthen}. To that end, we shall use Gauss' Series Convergence Test which ensures that if $a_n$ is a sequence of positive numbers such that 
\[
\frac{a_{n}}{a_{n+1} } \leq  1 + \frac{1}{n} + \frac{B_n}{n^r},
\]
for some $r>1$ and a bounded sequence $B_n$, then the series $\sum_{n=0}^\infty a_n$ is divergent. 

The formula \eqref{eq:hyp} implies
\begin{equation}\label{eq:1}
\frac 1 n + \frac c {n^r}  \ge \varrho_U(f^{n+1}(z), f^n(z)) = \varrho_\D(g^{n+1}(w), g^n(w))
\geq  \frac{2|g^{n+1}(w) - g^n(w)|}{|1 - g^n(w)\overline{g^{n+1}(w)}|}
\end{equation}
for $w \in \D$ and $z = \varphi(w)$. 
Since for any $u,v\in \D$, we have 
\[
|1-u \bar{v}| \leq 1-|v|^2 + \left||v|^2 - u \bar{v}\right| = 1- |v|^2 + |v||v-u| < 2 (1- |v|) + |v-u|,
\]
it follows that if we assume $|g^{n+1}(w)| > |g^n(w)|$, then
\begin{equation}\label{eq:2}
\begin{split}
\frac{2|g^{n+1}(w) - g^n(w)|}{|1 - g^n(w)\overline{g^{n+1}(w)}|} 
&\geq \frac{2|g^{n+1}(w) - g^n(w)|}{2(1 - |g^{n+1}(w)|) + |g^{n+1}(w)-g^n(w)|}\\
&= \frac{1}{(1 - |g^{n+1}(w))|/|g^{n+1}(w) - g^n(w)| + 1/2}\\
&\geq \frac{1}{(1 - |g^{n+1}(w))|/(|g^{n+1}(w)| - |g^n(w)|) + 1/2}\\
&= \frac{1}{a_{n+1}/(a_n - a_{n+1}) + 1/2}\\
&= \frac{1}{1/(a_n/a_{n+1} - 1) + 1/2},
\end{split}
\end{equation}
where
\[
a_n = 1 - |g^n(w)|.
\]
Note that $a_n > a_{n+1}$ by assumption, so $a_n/a_{n+1} - 1 > 0$. 

Using \eqref{eq:1}, \eqref{eq:2} and the fact $1 < r < 2$ we obtain
\begin{equation}\label{eq:a_n}
\frac{a_n}{a_{n+1}} \leq 1 + \frac{1}{1/(1/n + c/n^r) - 1/2} = 1 + \frac 1 n + \frac{cn + n^{r-1}/2 + c/2}{n(n^r-n^{r-1}/2 - c/2)} \le 1 + \frac 1 n + \frac{2c}{n^r}
\end{equation}
for large $n$ whenever $a_n > a_{n+1}$. Since \eqref{eq:a_n} holds trivially when $a_n \le a_{n+1}$, we conclude that \eqref{eq:a_n} is true for every sufficiently large $n$. 

Now by the Gauss series convergence test, \eqref{eq:a_n} implies that the series $\sum_{n=1}^\infty a_n = \sum_{n=1}^\infty(1 - |g^n(w)|)$ is divergent, so by Theorem~\ref{thm:ifthen}, the map $g$ on $\bdd$ is conservative with respect to the Lebesgue measure $\lambda$ on $\bdd$ (see Definition~\ref{defn:erg} and Theorem~\ref{thm:halmos}). Moreover, due to Theorems~\ref{thm:exact} and~\ref{thm:mu}, it is exact (in particular, ergodic) with respect to $\lambda$ and preserves the measure $\mu_p$ defined in \eqref{eq:mu}. By \eqref{eq:zero}, this implies that $g$ is non-singular with respect to $\lambda$ (see Definition~\ref{defn:erg}). 

Since $g$ is non-singular, conservative and ergodic with respect to the Lebesgue measure on $\bdd$, by Theorem~\ref{thm:conserv}, for every set $E \subset \bdd$ of positive Lebesgue measure, the forward trajectory under $g$ of Lebesgue-almost every point in $\bdd$ visits $E$ infinitely many times. Hence, for every set $B \subset \bd U$ of positive harmonic measure $\omega$, the forward trajectory under $f$ of $\omega$-almost every point in $\bd U$ visits $B$ infinitely many times. As the harmonic measure is positive on open sets in $\bd U$ (see e.g.~\cite{garnettmarshall}) and $\bd U$ is separable, for  
$\omega$-almost every point in $\bd U$ its forward trajectory under $f$ is dense in $\bdd$, which ends the proof of Theorem~C.

\section{Proof of Theorem~B and Proposition D} \label{sec:proofC}

Theorem~B follows immediately from Theorem~C. Indeed, it is enough to notice that if the associated inner function $g$ has a non-singular Denjoy--Wolff point $p \in \bdd$, then we can use the assertion~(c) of Theorem~\ref{thm:traj} to conclude that for $z \in U$ we have
\[
\varrho_U(f^{n+1}(z), f^n(z)) = \varrho_\D(g^{n+1}(w), g^n(w)) = \frac 1 {2n} + O\left(\frac{1}{n^{3/2}}\right), 
\]
so the assumption of Theorem~C is satisfied, which completes the proof of Theorem~B. 

\begin{rem}
An alternative proof of Theorem~B in the case when $f$ has finite degree on $U$ can be done by the use of the following result.

\begin{thm}[{\cite[Theorem \rm 6.1]{doering-mane}}]\label{thm:series}
Let $V$ be an invariant basin of a parabolic point $p \in \bd V$ of a rational map $R$ and let $F$ be a lift of $R$ by a universal covering $\pi: \D \to V$, i.e. $\pi \circ F = R \circ \pi$. Then for every $z \in \D$ and $\alpha > 1/2$,
\[
1 - |F^n(z)| \ge \frac{1}{n^\alpha}
\] 
for sufficiently large $n$.
\end{thm}

Applying Theorem~\ref{thm:series} to $R = g$, $V = \D$, one can check directly that the series $\sum_{n=1}^\infty (1 - |g^n(w)|)$ is divergent for $w \in \D$ and then proceed as in the final part of the proof of Theorem~C.
\end{rem}

We end this section by proving Proposition~D.

\begin{proof}[Proof of Proposition \rm D]
By the conformal change of coordinates $z\mapsto z/a$, we can assume $a = 1$. Then the assumption of the proposition has the form
\begin{equation}\label{eq:r}
|f(z) - z - 1| < \frac{c_0}{(\Re(z))^r} \qquad \text{for} \quad z \in H,
\end{equation}
where $H = \{w \in \C: \Re(w) > c_1\}$. Enlarging $c_1$, we can assume
$c_1 > 1/2$ and $\sum_{k = 1}^\infty c_0/(c_1 + k - 3/2)^r < 1/2$, which implies
\begin{equation}\label{eq:half}
\sum_{k = 1}^\infty \frac{c_0}{(\Re(z) + k - 3/2)^r} < \frac 1 2 \qquad
\text{for} \quad z \in H.
\end{equation}
Now we prove inductively that for every $n \ge 1$,
\begin{equation}\label{eq:ind}
\Re(f^n(z)) > \Re(z) + n - \sum_{k = 1}^n \frac{c_0}{(\Re(z) + k - 3/2)^r}
\qquad \text{for} \quad z \in H.
\end{equation}
To do it, note that \eqref{eq:ind} for $n = 1$ follows immediately from
\eqref{eq:r}. For $n > 1$ and $z \in H$ we obtain, using consecutively
\eqref{eq:r}, the inductive assumption and \eqref{eq:half},
\begin{multline*}
\Re(f^n(z)) > \Re(f^{n-1}(z)) + 1 - \frac{c_0}{(\Re(f^{n-1}(z)))^r} \\>
\Re(z) + n - \sum_{k = 1}^{n-1} \frac{c_0}{(\Re(z) + k - 3/2)^r} -
\frac{c_0}{\left(\Re(z) + n - 1 - \sum_{k = 1}^{n-1} c_0/(\Re(z) + k -
3/2)^r\right)^r}\\
> \Re(z) + n - \sum_{k = 1}^{n-1} \frac{c_0}{(\Re(z) + k - 3/2)^r} -
\frac{c_0}{(\Re(z) + n - 3/2)^r} \\= \Re(z) + n - \sum_{k = 1}^n
\frac{c_0}{(\Re(z) + k - 3/2)^r}.
\end{multline*}
which gives \eqref{eq:ind}. By \eqref{eq:ind} and \eqref{eq:half}, we have
\begin{equation}\label{eq:ind2}
\Re(f^n(z)) > c_1 + n - \frac 1 2 > n \quad \text{for} \quad z \in H,
\end{equation}
in particular $\Re(f^n(z)) \to +\infty$ for $z \in H$, so $H$ is contained
in an invariant Baker domain $U$ of $f$.

Suppose now that $U$ is simply connected. Then using \eqref{eq:hypdist}, \eqref{eq:r} and
\eqref{eq:ind2} we obtain, for $z \in H$ and large $n$,
\begin{multline*}
\varrho_U(f^{n+1}(z), f^n(z)) \le \varrho_H(f^{n+1}(z), f^n(z)) \le
\varrho_H(f^{n+1}(z), f^n(z) + 1) + \varrho_H(f^n(z) + 1, f^n(z))\\
\le \frac{2|f^{n+1}(z) - f^n(z) - 1|}{\min(\dist(f^{n+1}(z), \bd H),
\dist(f^n(z) + 1, \bd H))} + \ln \left(1 + \frac{1}{\Re(f^n(z)) -
c_1}\right) +  \\
\le \frac{c_0}{n^r(n+1 - c_1)} +  \frac{1}{n - c_1}  = \frac 1 n +
O\left(\frac{1}{n^2}\right)
\end{multline*}
as $n \to \infty$, so the assumptions of Theorem~{\rm C} are satisfied.
\end{proof}

\bibliography{escaping}

\end{document}